\documentclass{article}
\usepackage[utf8]{inputenc}
\newcounter{defn}
\newcounter{rmk}

\newcounter{conj}

\usepackage{amsmath}
\usepackage{float}

\usepackage{comment}
\usepackage[linktocpage=true]{hyperref}
\usepackage[margin=1.2 in, top=1 in, bottom= 1.2 in]{geometry}

\usepackage[english]{babel}
\usepackage{amsfonts}
\usepackage{mathrsfs}
\usepackage[mathscr]{euscript}
\usepackage{bbm}
\usepackage{latexsym}
\usepackage{math dots}
\usepackage{amssymb}
\usepackage{mathtools}
\usepackage{graphicx}
\usepackage{tikz}
\usepackage{placeins}

\usepackage{enumitem}
\setlist{nolistsep}
\usepackage{amsthm}
\usepackage{relsize}
\usepackage{nicefrac}
\usepackage[capitalize]{cleveref}

\newtheoremstyle{plain}{3mm}{3mm}{\slshape}{}{\bfseries}{.}{.5em}{}
\newtheoremstyle{definition}{2mm}{2mm}{}{}{\bfseries}{.}{.5em}{}
\theoremstyle{plain}
	
\newtheorem{theorem}{Theorem}

\newtheorem{conjecture}[conj]{Conjecture}

\theoremstyle{definition}
\newtheorem{definition}[defn]{Definition}
\newtheorem{remark}[rmk]{Remark}

\theoremstyle{plain}
\newtheorem*{namedthm}{\namedthmname}
\newcounter{namedthm}
	

\newcommand{\R}{\mathbb{R}}

\newcommand{\D}{\mathcal{D}}

\newcommand{\eps}{\epsilon}

\title{Weak geometric lemma for ADR boundary of a uniform domain implies uniform rectifiability.}
\author{Aritro Pathak}

\begin{document}

\maketitle

\begin{abstract}
We resolve a conjecture of Guy David and Svitlana Mayboroda. We show that if 
$\Omega \subset \mathbb{R}^n$ is a uniform domain with $(n-1)$-dimensional 
Ahlfors–David regular boundary $E = \partial \Omega$, and if $E$ satisfies the 
Weak Geometric Lemma, then $E$ satisfies the Bilateral Weak Geometric Lemma. 
In particular, $E$ is uniformly rectifiable.
\end{abstract}

\section{Introduction}

The following conjecture was posed by Guy David and Svitlana Mayboroda in Remark 6.10 of \cite{DM22}. 

\begin{conjecture}\label{conj}
Let $\Omega \subset \mathbb{R}^n$ be a domain with boundary 
$E = \partial \Omega$. Assume that $E$ is $(n-1)$-dimensional 
Ahlfors–David regular, and that $\Omega$ is a uniform domain. 
If $E$ satisfies the Weak Geometric Lemma, then $E$ satisfies 
the Bilateral Weak Geometric Lemma, and hence is uniformly rectifiable.
\end{conjecture}

By Corollary 2.10 of \cite{DS2}, the Bilateral Weak Geometric Lemma is equivalent to uniform rectifiability of the Ahlfors-David regular boundary. Note that no assumptions are made on the existence of exterior Harnack chains 
or exterior corkscrew points for $\Omega$.

In this paper, we prove the following.

\begin{theorem}\label{thm1}
Conjecture~\ref{conj} is true.
\end{theorem}

The theory of uniform rectifiability was studied in \cite{DS91,DS2}. The uniform rectifiability property (See \cref{defur}) for an a-priori Ahlfors-David regular set of codimension one (See \cref{defadr}) was shown in \cite{DS91} to be a sufficient condition for the $L^2$ integrability of all singular integral operators with odd, smooth kernels, the Riesz transform being one particular instance of such a singular integral operator. Later, the work of Nazarov-Tolsa-Volberg \cite{NTV12} also showed that uniform rectifiability of the a-priori Ahlfors-David regular boundary $E$ of codimension one, is implied by the $L^2(E)$ boundedness of the $(n-1)$ dimensional Riesz transform.  

Over the last few decades, there has also been a series of works on the equivalence of the regularity properties of the harmonic measure on rough boundaries, and the rectifiability properties of the a-priori Ahlfors-David regular boundary. It has been shown in \cite{HMMTZ21} that the weak $A_\infty$ property of elliptic measure for elliptic operators close to the Laplacian in a Carleson measure sense, in uniform domains, implies the uniform rectifiability of the boundary. \footnote{See \cite{HMMTZ21} for prior references as well, on the correspondence of absolute continuity of the harmonic measure versus just rectifiability of the boundary of a-priori 1 sided NTA domains.} The result of the present manuscript shows that under the background assumptions of quantifiable openness and quantifiable path connectedness of the domain, and a-priori regularity of certain measures supported on the boundary such as the harmonic measure, if one establishes the Weak Geometric Lemma, that is enough to conclude the uniform rectifiability of the boundary.

In \cite{DM22}, for the codimension one case, it is shown that if in a Carleson prevalent sense, the Green function is approximated by affine planes, then the boundary is uniformly rectifiable. Specifically in the proof of Theorem 6.1 in \cite{DM22}, using the affine approximation of the Green function by affine planes and the boundary Holder regularity of the Green function, one deduces the weak Geometric lemma. The result of this present paper shows that this immediately implies the boundary is uniformly rectifiable.

We state all the definitions of all the components of this conjecture. These can all be found, for example, in \cite{DS2} or in Section 2 of \cite{HMMTZ21}.

Throughout, let \(E \subset \mathbb{R}^{n}\) be an \(n\)-ADR set. Let \(\mathcal{D}(E)\) be a David--Semmes dyadic grid on \(E\). For \(Q \in \mathcal{D}(E)\), denote by \(\ell(Q)\) its side length, by \(x_Q\) the fixed center so that,
\[
B_Q := B(x_Q, C\,\ell(Q)),
\]
where \(C \) is a fixed constant.

We now define the Ahlfors-David regular boundary of codimension one, as well as interior corkscrews and interior Harnack chains.

\begin{definition}\label{defadr}(Ahlfors-David regular boundary)\label{def 2}
    The boundary $\partial\Omega$ is called Ahlfors-David $(n-1)$-regular, ($(n-1),C_1$ AD-regular) if
$\sigma = \mathcal{H}^{n-1}\!\!\mid_{\partial\Omega}$, then
\begin{equation}\label{adr}
C_{1}^{-1} r^{n-1} \le \sigma(B(x,r)) \le C_{1} r^{n-1}
\quad \text{for all } x \in \partial\Omega,\; 0 < r < \operatorname{diam}(\partial\Omega).\end{equation}

Here, we have $\sigma(B(x,r)):=\sigma(B(x,r)\cap \partial\Omega)$.

\end{definition}

\begin{definition}[Interior Corkscrews]\label{def1}
      An open set $\Omega\subset \R^{n}$ satisfies the interior corkscrew condition if for some uniform constant $c$ with $0<c<1$, and for every $x\in \partial \Omega$ and $0<r<\text{diam}\partial\Omega$, there is a ball $B(A(x,r),cr)\subset \Omega\cap B(x,r)$. The point $A(x,r)\in \Omega\cap B(x,r)$
 is called an interior corkscrew point relative to $\Delta$.
 \end{definition}

\begin{definition}[Interior Harnack Chains]\label{def3}
An open connected set $\Omega \subset \mathbb{R}^n$ is said to satisfy the 
Harnack chain condition with constants $M, C_1 > 1$ if for every pair of points 
$A, A' \in \Omega$ there is a chain of balls 
$B_1, B_2, ...., B_K \subset \Omega$ with 
\[
K \leq M\bigl(2 + \log_2^+ \Pi\bigr)
\]
that connects $A$ to $A'$, where
\begin{equation}\label{eq:Pi}
\Pi := \frac{|A - A'|}{\min\{\delta(A), \delta(A')\}}.
\end{equation}
Namely, $A \in B_1$, $A' \in B_K$, $B_k \cap B_{k+1} \neq \emptyset$ for every 
$1 \leq k \leq K$, and
\begin{equation}\label{eq:chain}
C_1^{-1} \,\mathrm{diam}(B_k) \leq \mathrm{dist}(B_k, \partial\Omega) 
\leq C_1 \,\mathrm{diam}(B_k).
\end{equation}

Alternately, the Harnack chain condition asserts there exists a uniform positive constant $M_0$ so that any two points $A,A' \in \text{int}\Omega$ can be connected by a non tangential path $\gamma$ with the property that for any point $z\in \gamma$, we have 
\begin{align}\label{eq4}
    \text{ min }(l_\gamma(A,z),l_\gamma(A',z))< M_0 \delta(z).
\end{align}
\end{definition}

Here, $\delta(z)$ denotes the distance of the point $z$ from the boundary $\partial\Omega$ and $l_\gamma(A,z)$ is the length along this path, between the points $A$ and $z$. 

This is a quantifiable path connectedness condition for $\Omega$; the interior corkscrew condition is a quantifiable openness condition for the set $\Omega$.

\begin{definition}[Uniform domain]\label{def4}
    An open set $\Omega\subset \R^n$ is a uniform domain if it satisfies the interior Harnack chains condition as well as the interior Corkscrew condition. 
\end{definition}

We now define the $\beta$ numbers under consideration.
\begin{definition}[Beta numbers $\beta_E$.]
    
\begin{equation}
\beta_E(Q)
:=
\inf_{P}
\sup_{x \in E \cap B_Q}
\frac{\operatorname{dist}(x,P)}{\ell(Q)},
\label{beta}
\end{equation}
where the infimum is taken over all \(n-1\)-planes \(P \subset \mathbb{R}^{n}\).
\end{definition}

\medskip

We define the bilateral \(\beta\)-numbers by
\begin{definition}[Bilateral beta numbers  $b\beta_E$]
\begin{equation}
b\beta_E(Q)
:=
\inf_{P}
\left(
\sup_{x \in E \cap B_Q}
\frac{\operatorname{dist}(x,P)}{\ell(Q)}
+
\sup_{y \in P \cap B_Q}
\frac{\operatorname{dist}(y,E)}{\ell(Q)}
\right),
\label{bbeta}
\end{equation}
where again the infimum is taken over all \(n-1\)-planes \(P \subset \mathbb{R}^{n}\).
\end{definition}

\medskip
\begin{definition}
 (Carleson sets): We borrow the definition of the Carleson sets from \cite{DS2} (see Definition 1.69 of \cite{DS2}).
\end{definition}

Let $E := \partial \Omega$ be an $(n-1)$-dimensional Ahlfors-David regular set in $\mathbb{R}^n$. A Carleson set is a measurable subset $A$ of $E \times \mathbb{R}_+$ such that $\chi_A(x,t)\, dx\, \dfrac{dt}{t}$ is a Carleson measure on $E \times \mathbb{R}_+$ (when $dx$ denotes $\mathcal{H}^{n-1}|_E$). This means that there is a $C>0$ such that,
$$\int_0^R \int_{E \cap B(w,R)} \chi_A(x,t)\, dx\, \frac{dt}{t} \leq C R^{n-1}, \quad \text{for all } w \in E \text{ and } R>0.$$

\begin{definition} (Weak Geometric Lemma): Let $E$ be a $d$-dimensional regular set. Then $E$ satisfies the weak geometric lemma (WGL) if for each $\varepsilon>0$, the set
$$\{(x,t) \in E \times \mathbb{R}_+ : \beta_E(x,t) > \varepsilon\}$$
is a Carleson set.
\end{definition}

\begin{definition} (Bilateral Weak Geometric Lemma) An $(n-1)$-dimensional Ahlfors-David regular set $E \subset \mathbb{R}^n$ is said to satisfy the bilateral weak geometric lemma (BWGL) if for each $\varepsilon > 0$,
$$\{(x,t) : E \times \mathbb{R}_+,\ b\beta_E(x,t) > \varepsilon\}$$
is a Carleson set.
\end{definition}

\begin{definition}\label{defur}  (Uniform rectifiability): An $(n-1)$ dimensional Ahlfors-David regular set $E \subset \mathbb{R}^n$ is said to be uniformly rectifiable, if it has big pieces of Lipschitz images of $\mathbb{R}^{n-1}$. That means, there exist $\theta, M>0$ such that for each $q \in E$, and $0 < r < \operatorname{diam}(E)$, there is a Lipschitz mapping $\rho : B_{n-1}(0,r) \to \mathbb{R}^n$ so that $\rho$ has Lipschitz norm $\leq M$, and
$$\mathcal{H}^{n-1}\big(E \cap B(q,r) \cap \rho(B_{n-1}(0,r))\big) \geq \theta r^{n-1}.$$
Here, $B_{n-1}(0,r)$ denotes a ball with center $0$, and radius $r$, in $\mathbb{R}^{n-1}$.
\end{definition}

\bigskip

Now we note that the definition for the weak geometric lemma and Bilateral weak geometric lemma, can be reduced to the following.

\begin{definition}[ Weak Geometric Lemma] An $(n-1)$ dimensional Ahlfors-David regular set $E \subset \mathbb{R}^n$ is said to satisfy the weak geometric lemma (WGL) if there exists some $\varepsilon_0 > 0$, so that for each $\varepsilon \leq \varepsilon_0$, the set
$$\{(x,t) \in E \times \mathbb{R}_+ : \beta_E(x,t) > \varepsilon\} \text{ is a Carleson set.}$$
\end{definition}

\begin{definition}[Bilateral Weak Geometric Lemma]
Similarly, $E$ satisfies the BWGL if there exists some $\varepsilon_0 > 0$, so that for each $\varepsilon \leq \varepsilon_0$, the set $\{(x,t) \in E \times \mathbb{R}_+ : b\beta_E(x,t) > \varepsilon\}$ is a Carleson set.
\end{definition}

\bigskip

These two statements follow, by noting that for any $\varepsilon_1 < \varepsilon_2$, we have
$$\{(x,t) \in E \times \mathbb{R}_+ : \beta_E(x,t) > \varepsilon_2\} \subset \{(x,t) \in E \times \mathbb{R}_+ : \beta_E(x,t) > \varepsilon_1\}$$
and thus if there exists such an $\varepsilon_0$ with the property that $\{(x,t) \in E \times \mathbb{R}_+ : \beta_E(x,t) > \varepsilon\}$ is a Carleson set when $\varepsilon \leq \varepsilon_0$, we also have the Carleson set property for any $1 > \varepsilon \geq \varepsilon_0$.

The same reasoning applies for the Bilateral Weak Geometric lemma (BWGL) as well.

 By a Euclidean cube $C(x,t)$ we mean the cube centered on $x$ with axes parallel to the coordinate axes, and so that the longest diagonal of this cube has length $2t$. 
 
 By covering a Euclidean cube $C(x,t)$ with the smallest ball $B(x,t)$ containing it and with the same center $x$, and conversely also considering the smallest cube $C(x,\sqrt{n}t)$ containing the ball $B(x,t)$ with the same center $x$, using the Ahlfors-David regular property for the case of the euclidean balls, we get that the Ahlfors-David regular property for the boundary is equivalent to the following stated in terms of Euclidean cubes.
 
 \begin{definition}\label{def13}(Ahlfors-David regular property for Euclidean cubes)
     The boundary $\partial\Omega$ is called Ahlfors-David $(n-1)$-regular, ($(n-1),C_1$ AD-regular) if
$\sigma = \mathcal{H}^{n-1}\!\!\mid_{\partial\Omega}$, then
\begin{equation}\label{adr}
C_{2}^{-1} r^{n-1} \le \sigma(C(x,r)) \le C_{2} r^{n-1}
\quad \text{for all } x \in \partial\Omega,\; 0 < r < \operatorname{diam}(\partial\Omega).\end{equation}
 \end{definition}

\textbf{Notation:} We use $\delta(x)$ to denote the distance from the boundary $\partial\Omega$. The cube $C(x,t)$ has center $x$ and has longest diameter equal to $2t$ and axes parallel to the coordinate axes.

In Case(ii) the proof of \cref{thm1}, we have used the notation $\vec{x}$ to denote the point $x$ as a vector. At other times we have used $x$ to denote the same point, and this will be clear from the context.  We also use the short form ADR to refer to either one of the inequalities in the Ahlford David regularity condition, that will be clear from the context. Throughout the proof of Case(ii), by the $(n-1)$ dimensional volume of the cube $C(x,t)$, we mean the $(n-1)$ dimensional volume of the intersection of $C(x,t)$ with a specific $(n-1)$ dimensional hyperplane containing $x$. We use $\alpha=\alpha(M)$ and $j=j(M)$ and it will always be understood that $\alpha,j$ depnds on the parameter $M$ that arises from the Harnack chain condition.

\section{Outline of the argument:} 

Assume to the contrary, there is a uniform domain $\Omega \subset \mathbb{R}^n$ with codimension one Ahlfors-David regular boundary $E = \partial \Omega$, so that $E$ fails the BWGL.Then there is a decreasing sequence of $\eps_i|_{i=0}^\infty$, $\eps_i \to 0$ as $i\to \infty$, with $\eps_0$ arbitrarily small. Thus the set
$$S = \left\{ (x,t) \in E \times \mathbb{R}_+ : b\beta_E(x,t) > \varepsilon_0 \right\}$$
fails to be a Carleson set, and so for any $C>0$ arbitrarily large, there exists $\omega \in E$, $R>0$, so that
$$\int_0^R \int_{B(\omega,R)} \chi_{S_i}(x,t)\, dx\, \frac{dt}{t} > CR^{n-1}.$$

Assume that $\beta_E(x,t) < \dfrac{\varepsilon_0}{4}$ for any pair $(x,t)\in S_i$ above.

Then there exists a plane $P_{(x,t)}$ so that
\begin{align}\sup_{y \in B(x,t) \cap E} \frac{d(y,P)}{t} < \frac{\varepsilon_0}{4} \qquad \end{align}

We also have, since $b\beta_E(x,t) > \varepsilon_0$, that,
\begin{align}\sup_{y \in B(x,t) \cap E} \frac{d(y,P)}{t} + \sup_{y \in B(x,t) \cap P} \frac{d(y,E)}{t} > \varepsilon_0 \qquad \end{align}

so, from $(1)$ and $(2)$ above we get
$$\sup_{y \in B(x,t) \cap P} \frac{d(y,E)}{t} > \frac{3\varepsilon_0}{4}>\frac{\eps_0}{2}.$$

There will thus be two possibilities: that there exists a ball $B(z,\eps_0 t/2)$ which belongs entirely either in the exterior of domain $\Omega^c$, or in interior of the domain $\Omega$. Using the interior corkscrew condition in the first case, and using the Harnack chain condition in the second case, we will then construct a counterexample to the ADR property for the cube $C(x,t)$. 

So, we must have, $\beta_E (x,t)\geq \eps_0/4$.

So, $t :=\{ (x,t)\in E\times R_{+}: \beta_E (x,t)\geq \eps_0 /4 \}\supset \{ (x,t)\in E\times R_{+}: b\beta_E (x,t)>\eps_0  \}$ fails to be a Carleson set. This argument is the same for each $\eps_i|_{i=0}^{\infty}$ and that concludes the argument.

\section{Proof of \cref{thm1}.}

\begin{proof}[Proof of \cref{thm1}] We prove the result by contradiction. Assume that the BWGL fails, for $E := \partial \Omega$.

Since BWGL fails, there exists a sequence $\eps_i|_{i=0}^{\infty}$ with  $\varepsilon_i \to 0$ so that the set $S_i = \{(x,t) \in E \times \mathbb{R}_+ : b\beta_E(x,t) > \varepsilon_i\}$ fails to be a Carleson set.

Let us fix an $\varepsilon_0$, to be chosen later, small enough. Thus the set $S_0 = \{(x,t) \in E \times \mathbb{R}_+ : b\beta_E(x,t) > \varepsilon_0\}$ fails to be a Carleson set, and so for any $C>0$ large enough there exists $w \in E$, $R>0$, so that
$$\int_0^R \int_{B(w,R)} \chi_{S_0}(x,t)\, dx\, \frac{dt}{t} > CR^{n-1}.$$

For any $(x,t) \in S_0$, consider the ball $B(x,t)$. Recall that the cube $C(x,t)$ has center $x$ and longest diameter equal to $2t$ and axes parallel to the coordinate axes.

Thus we have,
\begin{align}\label{eq9}
b\beta_E(x,t) = \inf_P \left\{ \sup_{y \in B(x,t) \cap E} \frac{d(y,P)}{t} + \sup_{y \in B(x,t) \cap P} \frac{d(y,E)}{t} \right\} > \varepsilon_0. \qquad \end{align}

\begin{figure}[h!]
\centering
\includegraphics[width=0.55\textwidth]{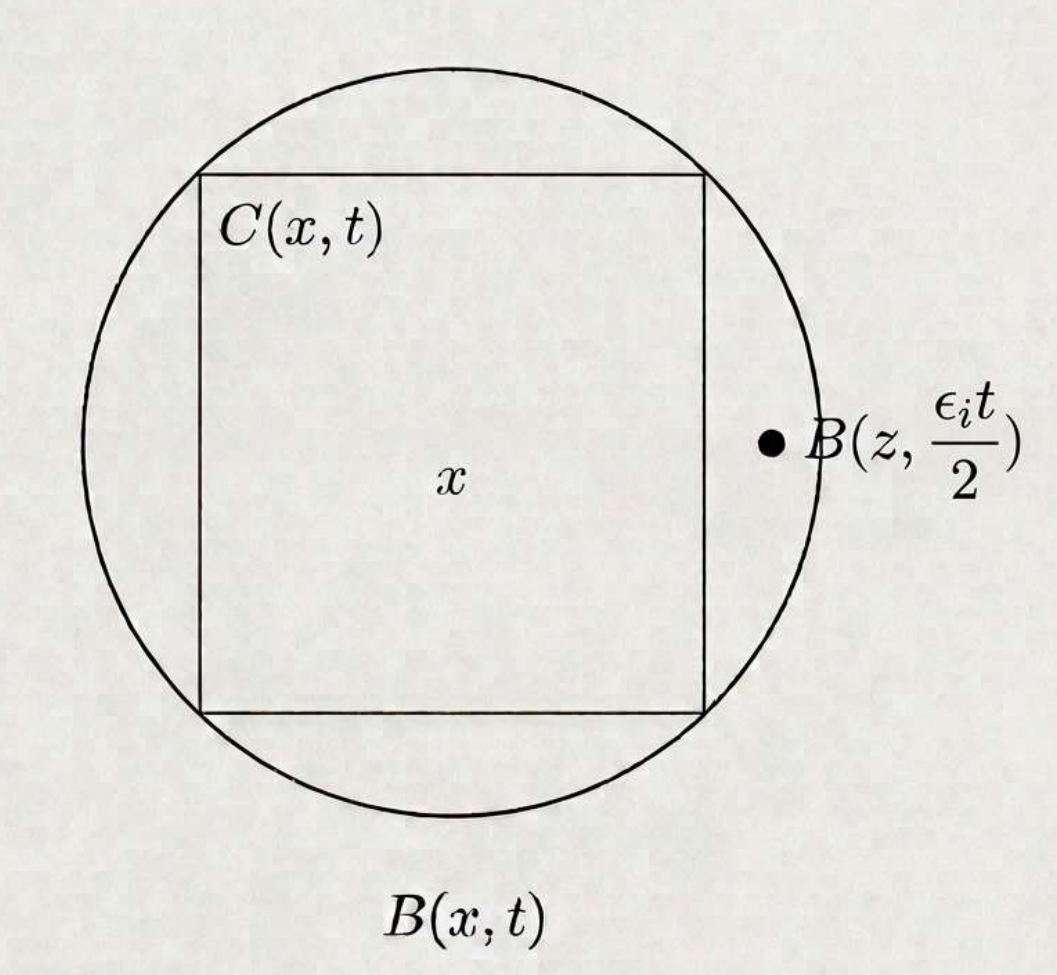}
\caption{The set up for the pair $(x,t)\in E\times \R_+$, where $\beta_E (x,t)\leq \eps_0 /2$ and where $b\beta_E(x,t)>\eps_0$. The cube $C(x,t)$ is the Euclidean cube centered on $x$, of maximal length, that is fully contained inside the ball $B(x,t)$. Due to the fact that $\beta_E (x,t)\leq \eps_0 /2$, we have a plane $P_{(x,t)}$ so that all points in $E\cap B(x,t)$ are $\eps_0 /2$ close to the plane $P_{(x,t)}$, and further due to the fact that $b\beta_E(x,t)>\eps_0$,  we also find a point $z\in B(x,t)\cap P_{(x,t)}$ so that $B(z,\frac{\eps_0 t}{2})\cap E=\phi$. Note that in general we have $B(z,\frac{\eps_0 t}{2})\not\subset B(x,t)$.}
\label{WGL}
\end{figure}

Assume that we have
\begin{align}\label{eq11}
    \beta_E(x,t) < \varepsilon_0/4.
\end{align}

So we have some plane $P_{(x,t)}$ so that
\begin{align}\label{eq10}
\sup_{y \in B(x,t) \cap E} \frac{d(y,P_{(x,t)})}{t} < \frac{\varepsilon_0}{4}. \qquad\end{align}

From \cref{eq9,eq10} above, since $(x,t) \in S_i$, we must have 
\begin{align}\sup_{y \in B(x,t) \cap P_{(x,t)}} \frac{d(y,E)}{t} > \frac{3\eps_0}{4}.\end{align}

Thus we have some point $z \in B(x,t) \cap P_{(x,t)}$ so that $d(z,E) > \dfrac{3\varepsilon_0 t}{4}>\dfrac{\varepsilon_0 t}{2}$. Thus there does not exist any point of $E$ belonging to $B\left(z, \dfrac{3\varepsilon_0 t}{4}\right)$. Now consider the ``slab" which is
\begin{align}\label{slab}
S_{x,t}:=\left\{ y \in C(x,t) : d(y,P_{(x,t)}) \leq \frac{\varepsilon_0 t}{4} \right\}.
\end{align}

By the standing assumptions, and the fact that no point of $E$ belongs to $B\left(z, \dfrac{3\varepsilon_0 t}{4}\right)$ and thus also no point belongs to $B\left(z, \dfrac{3\varepsilon_0 t}{4}\right)$, we get that
$$E \cap C(x,t) \subset \left\{ y \in C(x,t) : d(y,P_{(x,t)}) \leq \frac{\varepsilon_0 t}{4} \right\} \setminus B\left(z, \frac{\varepsilon_0 t}{2}\right). \qquad $$

Here, we have two possibilities:
\begin{enumerate}
\item[(1)] $B\left(z, \dfrac{\varepsilon_0 t}{2}\right) \subset \operatorname{ext}(\Omega)$
\item[(2)] $B\left(z, \dfrac{\varepsilon_0 t}{2}\right) \subset \operatorname{int}(\Omega)$.
\end{enumerate}

Note that we have chosen the values of $\eps_0 t/4$ for the width of the slab $S_{x,t}$ parallel to the plane $P_{}$ and $\eps_0 t/2$ for the radius of the ball $B\left(y, \dfrac{\varepsilon_0 t}{2}\right)$ with $y\in P_{(x,t)}$, so that the closure of the ball is not contained in the slab $S_{x,t}$, using the fact that no point of $E = \partial \Omega$ belongs to $B\left(y, \dfrac{\varepsilon_0 t}{2}\right)$, so that we precisely have the two possibilities as above; that the ball $B\left(y, \dfrac{\varepsilon_0 t}{2}\right)$ is contained entirely in either one of $\operatorname{ext}(\Omega)$ or $\operatorname{int}(\Omega)$.

\begin{enumerate}
    \item  We get in this case, that the complement of the slab $S_{x,t}=\{y \in C(x,t) : d(y,P_{(x,t)}) \leq \frac{\varepsilon_0 t}{4}\}$ within $C(x,t)$, also entirely belongs to $\operatorname{ext}(\Omega)$. This is seen by connecting any point of this aforementioned complement set
$$S^{c}_{x,t}=\{y \in C(x,t) : d(y,P_{(x,t)}) > \frac{\varepsilon_0 t}{4}\}$$
with $z$, by a continuous path that avoids the region $S_{x,t}\setminus B(z, \eps_0 t/2)$. Since this path cannot intersect $\partial \Omega$, we get that the entire path belongs to $\text{ext}( \Omega)$. Now since $x \in E$, consider the ball $B\left(x, \dfrac{\sqrt{\varepsilon_0} t}{4}\right)$. Since the domain $\Omega$ has interior corkscrews, we must have some interior ball
$$B\left(A\left(x, \frac{\sqrt{\varepsilon_0} t}{4}\right), \frac{c_0 \sqrt{\varepsilon_0} t}{4}\right) \subset B\left(x, \frac{\sqrt{\varepsilon_0} t}{2}\right)\cap \text{int}(\Omega)$$

If we choose $\varepsilon_0$ small enough so that
$$\frac{c_0 \sqrt{\varepsilon_0} t}{4} > \frac{\varepsilon_0 t}{2} \iff \varepsilon_0 < \left(\frac{c_0}{2}\right)^2,$$
then we cannot have the ball $B\left(A\left(x, \dfrac{\sqrt{\varepsilon_0} t}{2}\right)\right)$ contained in the slab $S_{x,t}=\{y \in C(x,t) : d(y,P_{(x,t)} \leq \frac{\varepsilon_0 t}{4}\}$. Here $c_0$ is the uniform constant arising from the interior corkscrew condition on the set $\Omega$. Thus we have $\text{int}(\Omega)$ intersects $S^{c}_{x,t}$.  This is a contradiction to the fact that $S^{c}_{x,t}$ belongs entirely to $\text{ext}(\Omega)$.

\item Now consider the case where $B(z, \frac{\varepsilon_0 t}{2})$ belongs entirely in $\text{int}(\Omega)$. Now we will carefully use the uniformity of the domain $\Omega$, along with the Ahlfors-David regularity of $E$ which is now contained entirely inside the slab 
$$S_{x,t}=\left\{ y \in C(x, t) : d(y, P_{(x, t)} )\le \frac{\varepsilon_0 t}{4} \right\}$$
defined earlier,  to again gain a contradiction.

We consider the $(n-1)$ dimensional hyperplane section $$\bar{P}(x, t) = P_{(x,t)} \cap B(x,t),$$ and consider a dyadic decomposition of this hyperplane. In the $j$-th stage of the dissection, we have $2^{j(n-1)}$ many disjoint congruent cubes each of side length $t/2^j$, so that their disjoint union is $\bar{P}(x, t)$. Call this set of cubes at this $j$-th dissection as $\mathcal{C}_j$, and so for each $j \ge 0$ we can write $\bar{P}(x, t)$ as a disjoint union of these cubes:
$$\bar{P}(x, t) = \bigcup_{C \in \mathcal{C}_j} C, \quad \text{for each } j \ge 0.$$

Now we use the uniformity of $\Omega$. We will choose $\varepsilon_0$ arbitrarily small in comparison to the Harnack chain parameters as well as the Ahlfors-David parameter, so that the slab is sufficiently thin for our purposes, in the entirety of the argument.





Since the assumption is that $\beta_E (x,t)< \eps_0 /4$, since $C(x, t) \subset B(x,t)$,
$$E \cap \left\{ y \in C(x, t) : d(y, P(x,t)) > \frac{\varepsilon_0 t}{4} \right\} = \phi.$$

We orient the axes so that $P_{(x,t)} = \{ (x_1, \dots, x_n) : x_n = 0 \}$, which is the $(n-1)$ dimensional hyperplane $x_n = 0$.

Now, consider the two points $\vec{x} + \alpha t \hat{e}_n$ and $\vec{x} - \alpha t \hat{e}_n$, for an $\alpha$ chosen small enough depending on the constants in the Harnack chain condition. Consider the Harnack chain joining these two points. The $\alpha$ purely depends on the Harnack chain constants so that the ensuing argument is satisfied, and we will inductively use this argument $N_0$ times, and then choose $\eps_0$ small enough in the end.

We choose $\alpha=\alpha(M)$ depending on $M$ to guarantee the existence of a ball $B(y, c t)$ with $y \in \bar{P}(x, t) \cap C\left(x, \frac{t}{2}\right)$ so that $E \cap B(y, c t) = \phi$, so that the Harnack path joining $\vec{x} + \alpha t x_n \hat{e}_n$ and $\vec{x} - \alpha t x_n \hat{e}_n$  intersects $B(y, c t)$ and  that this Harnack path remains quantifiably far from $E$ when it intersects $\overline{P}(x,t)$.  We have $|(\vec{x} + \alpha t \hat{e}_n) - (\vec{x} - \alpha t \hat{e}_n)| = 2 \alpha t$, $\delta_\Omega(\vec{x} + \alpha t \hat{e}_n) = \alpha t = \delta_\Omega(\vec{x} - \alpha t \hat{e}_n)$, where $\delta_\Omega (y)$ denotes the distance from the boundary $E$, of the point $y$. Then recalling the definition of the Harnack paths, we have that when this Harnack path joining $\vec{x} + \alpha t \hat{e}_n$ and $\vec{x} - \alpha t  \hat{e}_n$ intersects $\overline{P}(x,t)$, without loss of generality at the point $y$, and we get the ball $B(y,ct)$ with $B(y,ct)\cap E= 0$.

More precisely, we have $\Pi = 2$, and thus the admissible number of balls in the Harnack chain condition is upper bounded by $M \left(2 + \log_2^+ 2\right) = 3M$. Thus if we choose $\alpha$ small enough in comparison to $M$ which ensures that the Harnack path above has to be contained within $C(x,t/2)$, and then choose $\eps_0 \ll  \alpha(M). C(C_1,n,M)$ with a constant $C(C_1,n,M)=1/2^{N_0(C_1,n,M)(j(M)+1)}$ as given by \cref{epschoice},\footnote{ where the iteration runs $N_0 $ times.} and the ball of radius $c t$ with $c$ dependent on $M$, so that $E \cap B(y, c t) = \phi$ and $\eps_0 \ll c<_M \alpha$ where the rightmost inequality here follows from the Harnack inequality. 


Note also that we have also been able to choose $y \in C\left(x, \frac{t}{2}\right)$ with a crude factor of $\frac{1}{2}t$, so that $y$ stays a bounded distance away from the boundary of $C(x,t)$. We wish to extract a subcube of $C(x,t)$, centered on $\overline{P}(x,t)$ that does not contain any point of $E$. The Harnack chain condition only guarantees that there is a subball contained in $B_1 (x,t)$ centered on the same $x$, with $ B_1(x,t)\subset B(x,t)$, where the center of this subcube will lie within this $B_1(x,t)$. If $\alpha$ is large enough dependent on $M$ so that $B_1(x,t)
\setminus C(x,t)\neq \phi$, then in general we can't extract the required subcube of $C(x,t)$ which might now lie entirely outside $C(x,t)$. Taking the crude factor of $\frac{1}{2}t$ , requiring $\alpha$ small enough depending on $M$, we are able to extract the subcube from $C(x,\frac{1}{2}t)$.

Now consider the smallest $j = j(M) > 1$ so that $\frac{1}{2^j} < c$. Then consider the disjoint union of $\bar{P}(x, t)$ as $\bar{P}(x, t) = \bigcup_{C \in \mathcal{C}_{j(M)+1}} C$, so that there exists a cube $C(z, \frac{t}{2^{j+1}})$ so that $C(z, \frac{t}{2^{j+1}}) \subset B(y,c t)$ where $B(y,ct)$ is the ball considered above and where $B(y,ct)\cap E=\phi$. and where $z$ is the center of the cube $C(z, \frac{t}{2^{j+1}})$. We have the initial family $\mathcal{D}_0 = \left\{ C\left(z, \frac{t}{2^{j+1}}\right) \right\}$ which consists of just the single cube $C\left(z, \frac{t}{2^{j+1}}\right)$.

Now consider the collection of the cubes $\mathcal{F}_0 = \mathcal{C}_{j+1} \setminus C(z, \frac{t}{2^{j+1}})$, which is a disjoint covering of the plane $\bar{P}(x, t)$ excluding the specific cube $C(z, \frac{t}{2^{j+1}})$. Thus, the collection of cubes $\mathcal{C}_{j+1} \setminus C(z, \frac{t}{2^{j+1}})$ covers an $(n-1)$ dimensional volume $c_1 ( t^{n-1} - \left(\frac{t}{2^{j+1}}\right)^{n-1} )=c_1 t^{n-1} \left(1 - \frac{1}{2^{(j+1)(n-1)}}\right)$.

For each $C(z_1, \frac{t}{2^{j+1}}) \in \mathcal{C}_{j+1} \setminus C(z, \frac{t}{2^{j+1}})$, we consider the two points $\vec{z}_1 + \alpha \frac{t}{2^{j+1}} \hat{e}_n$ and $\vec{z}_1 - \alpha \frac{t}{2^{j+1}} \hat{e}_n$ and the Harnack chain connecting these two points. By an argument identical to the one prior, we get that there exists a ball $B\left(y_1, c \frac{t}{2^{j+1}}\right)$ with $y_1 \in C\left(z_1, \frac{t}{2 \cdot 2^{j+1}}\right)\cap \overline{P}(x,t)$, and so that there exists a cube $C\left(\bar{z}_1, \frac{1}{2^{j+1}} \cdot \frac{t}{2^{j+1}}\right) = C\left(\bar{z}_1, \frac{t}{2^{2j+2}}\right)$ with $C\left(\bar{z}_1, \frac{t}{2^{2j+2}}\right) \subset B\left(y_1, c \frac{t}{2^{j+1}}\right)$, and so that $B\left(y_1, c \frac{t}{2^{j+1}}\right) \cap E = \phi$, and consequently $C\left(\bar{z}_1, \frac{t}{2^{2j+2}}\right) \cap E = \phi$.

We repeat this argument of this previous paragraph for each cube belonging to $\mathcal{C}_{j+1} \setminus C\left(z, \frac{t}{2^{j+1}}\right)$.

Thus we have a set of cubes $\mathcal{D}_1 = \left\{ C\left(z_m, \frac{t}{2^{2j+2}}\right) \right\}$ where each $C\left(z_m, \frac{t}{2^{2j+2}}\right)$ is contained in a unique element of $\mathcal{F}_0 = \mathcal{C}_{j+1} \setminus C\left(z, \frac{t}{2^{j+1}}\right)$. 

Now consider the collection of cubes of radius $\frac{1}{2^{2j+2}}$, given by 
\begin{align}
\mathcal{F}_1 = \mathcal{C}_{2j+2} \setminus \left( \left\{ \bigcup_{D \in \mathcal{D}_1} D \right\} \cup \left\{ \bigcup_{\substack{G \in \mathcal{C}_{2j+2} \\ G \subset \mathcal{D}_0}} G \right\} \right)
\end{align}
Here, by abuse of notation, when we write $G\subset \D_0$, we mean that $G\subset \cup_{D\in \D_0} D$, and subsequently we define the same inductively for each $\D_k$ for $k\geq 1$.

In other words, $\mathcal{F}_1$ is the set of cubes of $\mathcal{C}_{2j+2}$, from which we exclude all the cubes of length $\frac{t}{2^{2j+2}}$ that were extracted in the previous step, and we also exclude the union of cubes in $\mathcal{C}_{2j+2}$ that are contained in the union of the cubes in $\mathcal{D}_0$.

Then we see that the disjoint union of cubes in $\mathcal{F}_1$ has the total volume 
\begin{multline}\label{eq15}
c_1 \Big( t^{n-1} - \left(\frac{t}{2^{j+1}}\right)^{n-1} \Big) \left(1 - \left(\frac{1}{2^{j+1}}\right)^{n-1}\right)  \\= c_1 t^{n-1} \left(1 - \frac{1}{2^{(j+1)(n-1)}}\right)^2 =|C(x, t)| \left(1 - \frac{1}{2^{(j+1)(n-1)}}\right)^2.
\end{multline}
This follows by noting that the cubes of $\mathcal{F}_1$ consists of the all cubes in $\mathcal{C}_{2j+2}$ , first excluding one cube from the generation $\D_0$ of volume $c_1 (t/2^{j+1})^{n-1}$after which a set of  volume of $c_1 t^{n-1} \left(1 - \frac{1}{2^{(j+1)(n-1)}}\right)$ remains. This set is now split up into $2^{(j+1)(n-1)}-1$ congruent disjoint cubes,  each of volume $c_1 (t/2^{j+1})^{n-1}$,  belonging to $\mathcal{C}_{j+1}$. In the process described above, from each one of these ($2^{(j+1)(n-1)}-1$ many ) cubes, we extract a subcube belonging to $\mathcal{C}_{2j+2}$ of volume $c_1 (t/2^{2j+2})^{n-1}$ and thus a fraction $\Big(1-\frac{1}{2^{(j+1)(n-1)}}\Big)$ of the volume of each one of these cubes survives, and finally the union of cubes in $\mathcal{F}_1$ has volume given in \cref{eq15}.

Now, we repeat this process, for each cube $C\left(z_{11}, \frac{t}{2^{2j+2}}\right) \in \mathcal{F}_1$, considering the two points $\vec{z}_{11} + \alpha \frac{t}{2^{2j+2}} \hat{e}_n$ and $\vec{z}_{11} - \alpha \frac{t}{2^{2j+2}} \hat{e}_n$ and the Harnack chain connecting them. Note again that we have chosen $z_{11}\in \overline{P}(x,t)$. This gives us a ball $B\left(y_{11}, c \frac{t}{2^{2j+2}}\right)$ with $y_{11} \in C\left(z_{11}, \frac{t}{2 \cdot 2^{2j+2}}\right)$, and so that there exists a cube $C\left(\bar{z}_{11}, \frac{t}{2^{3j+3}}\right)$ so that $C\left(\bar{z}_{11}, \frac{t}{2^{3j+3}}\right) \cap E = \phi$, and $C\left(\bar{z}_{11}, \frac{t}{2^{3j+3}}\right) \in \mathcal{C}_{3j+3}$.

Repeating this process for all the cubes in $\mathcal{F}_1$, we extract a set of cubes 
\begin{align}\mathcal{F}_2 = \mathcal{C}_{3j+3} \setminus \left( \left\{ \bigcup_{D \in \mathcal{D}_2} D \right\} \cup \left\{ \bigcup_{\substack{G \in \mathcal{C}_{3j+3} \\ G \subset \mathcal{D}_1}} G \right\} \cup \left\{ \bigcup_{\substack{J \in \mathcal{C}_{3j+3} \\ J \subset \mathcal{D}_0}} J \right\} \right)\end{align}
where $\mathcal{D}_2$ is the set of cubes $\mathcal{D}_2 = \left\{ C\left(\bar{z}_m, \frac{t}{2^{3j+3}}\right) \right\}$ where each such cube belonging to $\mathcal{D}_2$ is extracted from a unique element of $\mathcal{F}_1$.

Again, the disjoint union of the set of cubes in $\mathcal{F}_2$ have the total volume $|C(x, t)| \left(1 - \frac{1}{2^{(j+1)(n-1)}}\right)^3$.

We thus inductively define the collections of cubes $\mathcal{D}_k, \mathcal{F}_k$ for each $k \ge 3$ as well, and the set of cubes in each $\mathcal{F}_k$ has total volume $|C(x, t)| \left(1 - \frac{1}{2^{(j+1)(n-1)}}\right)^{k+1}$.

We continue this process inductively, and if we take $k$ large enough, we have a contradiction to the Ahlfors regularity lower bound for the surface measure contained in $C(x, t)$, as argued below.

This total volume of $\mathcal{F}_k$ is distributed among a disjoint set of $(2^{(j+1)(n-1)}-1)^{k}$ many cubes each of length $\frac{t}{2^{(j+1)k}}$, and using \cref{def13} on each one these disjoint cubes, and then adding the contribution, we get, 
$$\mathcal{H}^{n-1} \left( E \cap C(x, t) \right) \le C_1 |C(x, t)\cap \overline{P}(x,t)| \left(1 - \frac{1}{2^{(j+1)(n-1)}}\right)^{k+1}.$$

On the other hand, we also have due to Ahlfors-David regularity for the cube $C(x, t)$, that
$$\mathcal{H}^{n-1} \left( E \cap C(x, t) \right) \ge C_1^{-1} |C(x, t)\cap \overline{P}(x,t)|,$$
where $|C(x, t)|=c_1  t^{n-1}$ is the $(n-1)$ dimensional volume of the cube $C(x, t) \subset P_{(x,t)}$. Noting that the fixed $j=j(M)$ being used here is dependent on the Harnack chain constants, we will get a contradiction when
\begin{align}C_1 \left(1 - \frac{1}{2^{(j(M)+1)(n-1)}}\right)^{k+1} \le C_1^{-1}, \quad \text{so,} \quad \left(1 - \frac{1}{2^{(j(M)+1)(n-1)}}\right)^{k+1} \le C_1^{-2}.\end{align}

Thus, we can choose the minimal possible $k:=N_0$, dependent on $C_1, n,j(M)$ (recall that $j(M)$ is an exponent depending on the Harnack chain constants), so that the above inequality holds.

Thus the final set of cubes that have been considered, belong to $\mathcal{C}_{k(j+1)}$, whose cubes have length $\frac{t}{2^{N_0(j(M)+1)}}$. Thus finally we choose the 
\begin{align}\label{epschoice}
\varepsilon_0  \ll \alpha(M)\frac{1}{2^{(N_0(C_1,n,M))(j(M)+1)}}
\end{align} 

(where in turn as mentioned above, $N_0$ depends on $C_1, n, j(M)$,) so that at each stage, the Harnack chain condition allows us to choose the cubes centered on $\overline{P}(x,t)$ that belong to $\mathcal{D}_p$, for each $1 \le p \le N_0$, since the width $\frac{\varepsilon_0 t}{4} \ll \frac{\alpha}{2^{N_0(j(M)+1)}}t< \frac{\alpha}{2^{l(j(M)+1)}}t$ for any $1 \le l \le N_0$.
\bigskip

Thus we have a contradiction to the Ahlfors regularity for the cube $C(x, t)$ in Case(ii), and thus we must have that the pair $(x, t) \in \left\{ (y, s) : \beta_E(y, s) > \frac{\varepsilon_0}{4} \right\}$. 

\end{enumerate}

We also had a contradiction in Case(i).

Therefore the failure of BWGL implies that given the $\eps_0$, for any $C_0$ we also have a $w \in E$ and $R > 0$, so that
$$\int_0^R \int_{B(w, R)} \chi_{T}(x, t) \, dx \, \frac{dt}{t} > C_0 R^{n-1}, \quad \text{where } T = \left\{ (x, t) : \beta_E(x, t) > \frac{\varepsilon_0}{4} \right\}.$$


Thus, we cannot have a set $E = \partial\Omega$, which satisfies the ADR property, where $\Omega$ is a uniform domain, where $\partial\Omega$ fails the BWGL property, and for which the $\frac{\varepsilon_0}{4}$ WGL condition holds. 

Taking the sequence $\eps_i\to 0$ and repeating the argument, with each $\eps_i \leq \eps_0$, this concludes the proof of the theorem by contradiction.\end{proof}

\begin{remark}
Note that if our domain also had the corkscrews for the exterior domain $\Omega^c$ as well (which is precisely the interior corkscrew condition for the domain $\Omega^c$), the analysis for the situation in Case(ii) would also reduce exactly to the proof used in Case(i). In fact, the result of \cite{DJ90} shows the uniform rectifiability of a codimension one ADR boundary which contains both interior and exterior corkscrews . On the other hand, in \cite{AHMNT14} it is shown that if the codimension $1$ boundary of a uniform boundary is uniformly rectifiable (which already means it satisfies the ADR property), then the domain contains exterior corkscrews as well.

\end{remark}

\begin{remark}
    In each of the $N_0$ steps, because of the Harnack chain condition, we have been able to extract a cube centered on the hyperplane section $\overline{P}(x,t)$, which does not intersect $E$, and where the radius of the extracted cube at each of these $N_0$ steps is comparable to $t$. This enabled us to get a contradiction to the ADR lower bound applied to the cube $C(x,t)$. Note, however, that if we considered the Harnack path joining points of the form $\vec{x}+\alpha_1 t \hat{e_n}$ for some $\alpha_1$ comparable to $\eps_0$, then we can only say using \cref{slab}  that there are some arbitrarily small subballs contained within the slab $S_{x,t}$, not necessarily centered on $\overline{P}(x,t)$, which do not contain points  of $E$. Then the contradiction to the ADR lower bound for the original cube $C(x,t)$ will not follow by the same argument.
\end{remark}

\section{Counterexample in absence of uniformity of domain, from \cite{DS91}.}

 Note that the condition $\{b\beta_{E}(x,t)>\eps\}$ in conjunction with the condition $\{\beta_{E}(x,t)<\eps/4\}$, for any pairs of values belonging in $X\times T$, for a sequence $T= \{t_i \}|_{i=1}^\infty\ \text{with } t_i \to 0$ as $i\to \infty$, $t_{i}/t_{i+1}\to \infty \text{ as } i\to \infty$, and any bounded sequence of  points $X=\{x_i:x_i\in B(q,r_0)\}$ for some fixed $r_0$ and some fixed $q\in E$ , in a manner that ensures the failure of the BWGL Carleson packing condition, forces us to find ``holes" in the boundary ball $\partial\Omega\cap E$, but at radial scales whose successive ratios are arbitrarily large as taken above. 
 
 One sees that this alone cannot force the failure of the ADR property for the cubes  $C(x, t)$ in the manner of the argument of the proof of \cref{thm1}. In particular, in the proof of \cref{thm1},  for any pair $(x,t)$ with the conditions that $\{b\beta_{E}(x,t)>\eps\}$ and $\{\beta_{E}(x,t)<\eps/4\}$, in case (ii), we find holes in the boundary layer for a set of radial scales whose successive ratios are fixed, with each of the radial values less than and comparable to $t$,  that leads to the failure of the ADR condition for the cube $C(x,t)$. 

 Indeed, one can still have a codimension one ADR boundary of a domain that is not uniform, with the failure of the BWGL condition, while the WGL property is still preserved for this boundary.

One is referred to the counterexample in Section 20 of \cite{DS91} for an example of a domain that is not uniform, and whose boundary satisfies the WGL property, but still fails the BWGL property, where the boundary is $1$-dimensional Ahlfors-David regular in $\R^2$. 

In \cite{DS2, DS91} it is also shown that the Weak Geometric Lemma with a big projections condition implies the Bilateral Weak Geometric Lemma for a codimension one set $E\subset \R^n$, with no uniformity required for the domain $\R^n \setminus E$ .

\section{Acknowledgements}
The author is grateful for the support of the School of Mathematics, Tata Institute of Fundamental Research in Mumbai, where this work was completed. The figure in this manuscript was generated using Google Gemini's Large Language Model.


\bigskip

Email: ap7mx@missouri.edu.

\end{document}